\documentclass[12pt, a4paper]{article}
\usepackage{natbib}
\usepackage{authblk}
\usepackage{centernot}
\usepackage{tikz-qtree}
\usepackage{color}

\usepackage{amsmath}
\usepackage{amssymb}
\usepackage{amsthm}
\usepackage{graphicx}
\usepackage{enumerate}
\usepackage{mathtools}

\newcommand{\fQ}{\mathbb{Q}}
\newcommand{\fZ}{\mathbb{Z}}

\newcommand{\fN}{\mathbb{N}}

\theoremstyle{plain}
\newtheorem{stam}{STAM}[section]
\newtheorem{lem}[stam]{Lemma}
\newtheorem{thm}[stam]{Theorem}
\newtheorem{prop}[stam]{Proposition}

\newtheorem{cor}[stam]{Corollary}

\newtheorem*{lem*}{Lemma}
\newtheorem*{thm*}{Theorem}
\newtheorem*{prop*}{Proposition}
\newtheorem*{claim*}{Claim}
\newtheorem*{cor*}{Corollary}
\newtheorem*{conj*}{Conjecture}
\newtheorem*{obs*}{Observation}
\theoremstyle{definition}
\newtheorem{definition}[stam]{Definition}

\newtheorem*{definition*}{Definition}
\newtheorem*{notation*}{Notation}

\theoremstyle{remark}
\newtheorem{rem}[stam]{\textbf{Remark}}
\newtheorem*{rem*}{\textbf{Remark}}

\newtheorem*{exercise*}{\textbf{Exercise}}
\numberwithin{equation}{section}

\newcommand{\norm}[1]{\left\vert #1 \right\vert}

\newcommand{\set}[1]{\left\{ #1 \right\}}

\newcommand{\floor}[1]{\lfloor #1 \rfloor}
\newcommand{\tuple}[1]{\langle #1 \rangle}

\DeclareMathOperator{\sign}{sign}

\begin{document}
\title{Effective Martingales with Restricted Wagers}

\author{Ron Peretz\thanks{Supported in part by the Google Inter-university Center for Electronic Markets and Auctions.\\ronprtz@gmail.com}}
\affil{London School of Economics}
\maketitle

\begin{abstract}
The classic model of computable randomness considers martingales that take real or rational values. Recent work by \cite{teutsch-etal2012} and \cite{teutsch2013} shows that fundamental features of the classic model change when the martingales take integer values. 

We compare the prediction power of martingales whose wagers belong to three different subsets of rational numbers: (a) all rational numbers, (b) rational numbers excluding a punctured neighborhood of 0, and (c) integers. We also consider three different success criteria: (i) accumulating an infinite amount of money, (ii) consuming an infinite amount of money, and (iii) making the accumulated capital oscillate.

The nine combinations of (a)--(c) and (i)--(iii) define nine notions of computable randomness. We provide a complete characterization of the relations between these notions, and show that they form five linearly ordered classes.  

Our results solve outstanding questions raised in \cite{teutsch-etal2012}, \cite{chalcraft12}, and \cite{teutsch2013}, and strengthen existing results.
\end{abstract}

\section{Introduction}
\subsection{Restricted wagers and effective prediction}

A binary sequence that follows a certain pattern can serve as a test for the sophistication of gamblers. Only martingales that are sufficiently ``smart'' should be able to recognize the pattern and exploit it. Conversely, a martingale (or a class of martingales) can serve as a test for predictability. A predictable sequence is one that can be exploited by that martingale (or class of martingales). When we consider the class of all recursive martingales, unpredictable sequences are called computably random.

Following our intuitive notion of randomness, when a martingale (or a countable class of martingales) bets against the bits of a random binary sequence, its accumulated capital should (almost surely) converge to a finite value. So, a ``predictable'' sequence should be defined as one on which that martingale (or at least one martingale in that class) does not converge. Not converging divides into two cases: going to infinity, and oscillating. The former is used as the \emph{success criterion} in the classic definition of ``computable randomness''; we call it $\infty$-\textsc{gains}. The latter we call \textsc{oscillation}. A third, economically appealing, success criterion requires that the martingale specify a certain amount to be consumed at each turn and the accumulated consumption go to infinity. A martingale together with a consumption function describe a \emph{supermartingale}. We call the success of a supermartingale $\infty$-\textsc{consumption}.

Things are not very interesting unless restricted wagers are introduced. It turns out that the above three success criteria are equivalent when all rational-valued martingales\footnote{Some papers consider computable real-valued martingales. These have the same prediction power as computable rational-valued martingales by simple approximation.} are allowed. Matters become more involved when the wagers of martingales are restricted to subsets of the rationals. We consider three sets of wagers: $\fQ$, $V=\{x\in\fQ:\ |x|\geq 1\text{ or }x=0\}$, and $\fZ$. The wager sets together with the success criteria form nine \emph{predictability classes}. A complete characterization of the relations between these classes is given (see Figure~\ref{fig summary}).

\subsection{Relations to existing literature}
Computable randomness was introduced by \cite{schnorr71}. For background, see also \cite{downey-online}, \cite{downey-book}, or \cite{nies-book}. 
The present paper is motivated by refinements of the notion of computable randomness recently introduced by \cite{teutsch-etal2012}, \cite{chalcraft12}, and \cite{teutsch2013}. The main notion of computable randomness is $\infty$-\textsc{gains}. Other well-studied success criteria include those of \cite{schnorr71} and \cite{kurtz81}. The less familiar success criteria, $\infty$-\textsc{consumption} and \textsc{oscillation}, turned out to be equivalent to the main notion of computable randomness (as mentioned above) and became folklore. When \cite{teutsch-etal2012} introduced integer-valued martingales some of the folklore criteria gained renewed interest. 

\cite{teutsch-etal2012} showed that rational- and integer-valued martingales (rational and integer martingales, for short) are different with respect to $\infty$-\textsc{gains}. Section~\ref{sec casino} provides an alternative elementary proof. Theorem~\ref{thm R gain not to V gain} shows that the separation can be done with a very simple history-independent martingale.

\cite{teutsch-etal2012} asked whether martingales whose wagers take values in $V$ (defined above) were different from integer martingales with respect to $\infty$-\textsc{gains}. Theorem~\ref{thm V gain not to Z gain} answers their question in the affirmative.

\cite{teutsch2013} introduced the success criterion we call $\infty$-\textsc{consumption} as a \emph{qualitative} distinction between rational and integer martingales. He showed\footnote{Modulo a minor mistake that is corrected here.} that $\infty$-\textsc{gains} and $\infty$-\textsc{consumption} are equivalent for rational but not for integer martingales. He asked what the relation is between $\infty$-\textsc{gains} and $\infty$-\textsc{consumption} for $V$-martingales. Proposition~\ref{prop V gain to save} shows that the two are equivalent for $V$-martingales. However, \textsc{oscillation} can serve as a qualitative distinction between rational and $V$-martingales. Propositions~\ref{prop Q to I gain} and \ref{prop I gain to oscillation} and Theorem~\ref{thm V save not to oscillate} show that $\infty$-\textsc{consumption} and \textsc{oscillation} are equivalent for rational but not for $V$-martingales.

\cite{teutsch-etal2012} used Baire category to distinguish between integer and $V$-martingales to rational martingales. Their proof is based on the observation that the former attain local minima, since they stop once they get close to a local infimum, whereas the latter may decrease in very small steps indefinitely. Baire's category seems to be too coarse to distinguish between two sets that exclude a punctured neighborhood of zero, such as $V$ and $\fZ$. \cite{chalcraft12} introduced a different argument when they characterized the relations between finite wager sets with respect to $\infty$-\textsc{gains}. They asked whether their characterization extends to infinite sets. Proposition~\ref{prop Q to I gain} answers this question in the negative.

\section{The casino setting}\label{sec casino}

Before providing the formal definitions, we first consider an illustrative example that demonstrates how to distinguish between classes of predictability (specifically, integer and rational martingales with respect to $\infty$-\textsc{gains}).
   
A sequence of gamblers enter a casino. Gambler 1 declares her betting strategy, a function from finite histories of Heads and Tails to rational-valued bets. Then the rest of the gamblers, 2, 3,\ldots (countably many of them), declare their strategies, which are restricted to integer-valued bets. The casino wants Gambler 1 to win and all the others to lose. That is, the casino should choose a sequence of Heads and Tails so that the limit of Gambler 1's capital is infinite and everyone else's is finite.  

Is it possible? Consider the following strategy for Gambler 1. She enters the casino with two cents in her pocket (any non-dyadic fraction of a dollar will do). After $t$ periods she has $x_t$ dollars in her pocket and she bets $\frac 1 2 \set{x_t}$ on Heads (where $\set{x}:=x-\floor{x}$).

Now, Gamblers 2, 3,\ldots declare their betting strategies. The casino places a finite sequence $\sigma$ of Heads and Tails on which the capital of Gambler 2 is minimal among all such finite sequences. Recall that Gamblers 2, 3,\ldots may bet only integer numbers; hence that minimum exists. At this point Gambler 2 is bankrupt. If he places a (non-zero) bet afterwards, it will contradict the fact that $\sigma$ is a minimizer of his capital.\footnote{This argument extends to $V$-martingales. \citet[Lemma 4]{teutsch-etal2012} essentially asserts that for any $V$-martingale, any string $\sigma$ has an extension $\tau$ such that the martingale is constant on extensions of $\tau$.} Note that Gambler 1 bets only on the fractional part of her capital, and so $\floor{x_t}$ never decreases. 

In the next stage, the casino extends $\sigma$ by appending to it sufficiently many Heads, so that Gambler 1's capital increases by at least 1. The casino repeats the same trick against every gambler in turn in order to bankrupt him while ensuring that Gambler 1 does not lose more than the fractional part of her capital, and then it continues to place Heads until she accumulates a dollar. QED.

\section{Definitions}
The set of all finite bit strings is denoted $\set{-1,+1}^{<\infty}=\bigcup_{n=0}^\infty\set{-1,+1}^n$. The length of a string $\sigma\in\set{-1,+1}^{<\infty}$ is denoted $\norm{\sigma}$. The empty string is denoted $\varepsilon$. The concatenation of two strings $\sigma$ and $\tau$ is denoted $\sigma,\tau$. For an infinite bit sequence $x\in\set{-1,+1}^\fN$ and a non-negative integer $n$, the prefix of $x$ of length $n$ is denoted $x\restriction n$.
 
A \emph{supermartingale} is a function $M:\set{-1,+1}^{<\infty}\to \fQ$, satisfying
\[
M(\sigma)\geq \frac{M(\sigma,-1)+M(\sigma,+1)}2,
\]
for every $\sigma\in\set{-1,+1}^{<\infty}$.
\begin{rem}
Restricting supermartingales to rational values (rather than allowing all real values) is meant to avoid unnecessary technicalities that would arise from considering real-valued computable functions. This restriction does not result in a loss of generality, as the questions addressed in the present paper are such that any real-valued supermartingale could be approximated by a rational-valued one. 
\end{rem}

We call the difference $M(\sigma)- \frac 1 2\left(M(\sigma,-1)+M(\sigma,+1)\right)$ $M$'s \emph{marginal consumption} at $\sigma$. If $M$'s marginal consumption is 0 at every $\sigma\in\set{-1,+1}^{<\infty}$, we say that $M$ is a \emph{(proper) martingale}. 

The next couple of paragraphs contain definitions that are exemplified in Figure~\ref{fig example}. The \emph{wager} of $M$ at $\sigma$ is defined as
\[
M'(\sigma) = \frac{M(\sigma,+1)-M(\sigma,-1)}2.
\]
Note that $M'(\sigma)$ is positive if $M$ bets on ``$+1$'' and negative if $M$ bets on ``$-1$'' at $\sigma$. When $M$ is a proper martingale, our definition of wager coincides with the classic definition
\[
M'(\sigma)=M(\sigma,+1)-M(\sigma).
\]
The \emph{initial capital} of (a supermartingale) $M$ is defined as $M(\varepsilon)$. A proper martingale is determined by its initial capital and its wagers at every $\sigma\in\set{-1,+1}^{<\infty}$. 

For a supermartingale $M$, the \emph{proper cover} of $M$ is the martingale $\tilde M$ whose initial capital and wagers are the same as $M$'s. The \emph{accumulated consumption} of $M$ is defined as $\tilde M -M$. Note that the accumulated consumption of $M$ at $\sigma$ is the sum of $M$'s marginal consumption over all proper prefixes of $\sigma$.
\begin{figure}
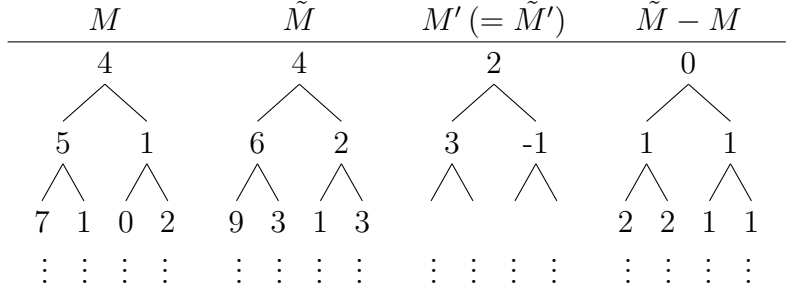

\begin{center}
\begin{tabular}{cccc}
$M$ & $\tilde M$ & $M'\, (=\tilde M')$ & $\tilde M- M$\\
\hline
\Tree [.{4}  
         [.5 7 1 ] 
         [.1 0 2 ]
      ] &

\Tree [.{4}  
         [.6 9 3 ] 
         [.2  1 3 ]
      ] &
\Tree [.{2}  
         [.3  {\color{white} 1} {\color{white} 1} ] 
         [.-1 {\color{white} 1} {\color{white} 1} ]
      ] &
\Tree [.{0}  
         [.1 2 2 ] 
         [.1  1 1 ]
      ]\\
$\vdots\quad \vdots\quad \vdots\quad \vdots$ &$\vdots\quad \vdots\quad \vdots\quad \vdots$ & $\vdots\quad \vdots\quad \vdots\quad \vdots$ & $\vdots\quad \vdots\quad \vdots\quad \vdots$
\end{tabular}
\end{center}
\caption{\label{fig example} A supermartingale $M$ with its proper cover $\tilde M$, wager $M'$, and accumulated consumption functions  $\tilde M -M$. The tree nodes represent finite strings with left descendants corresponding to $+1$ extensions.}
\end{figure}

For a supermartingale $M$ and a string $\sigma$, we say that $M$ \emph{goes bankrupt at} $\sigma$, if the sum of $M$'s wager and marginal consumption exceeds $M$'s capital. That is,
\[
M(\sigma)-\norm{M'(\sigma)}< M\text{'s marginal consumption at }\sigma.
\]
We say that $M$ goes bankrupt at $\sigma$ if incurring  a loss in the next round following $\sigma$ will make $M$'s capital negative. Consequently, $M$ never goes bankrupt if and only if it is always non-negative. 

For an infinite sequence $x\in\set{-1,+1}^\fN$, we say that $M$ \emph{goes bankrupt on} $x$, if $M$ goes bankrupt at $x\restriction n$, for some non-negative integer $n$. 
   
Let $M$ be a supermartingale and $x\in\set{-1,+1}^\fN$. If $M$ does not go bankrupt on $x$, we say that $M$ achieves
\begin{itemize} 
\item \emph{$\infty$-\textsc{gains}} on $x$, if $\lim_{n\to\infty} M(x\restriction n)=\infty$;
\item \emph{$\infty$-\textsc{consumption}} on $x$, if $\lim_{n\to\infty}(\tilde M-M)(x\restriction n)=\infty$;
\item \emph{\textsc{oscillation}} on $x$, if $\liminf_{n\to\infty} \tilde M(x\restriction n) \neq \limsup_{n\to\infty} \tilde M(x\restriction n)$.
\end{itemize}
We refer to $\infty$-\textsc{gains}, $\infty$-\textsc{consumption}, and \textsc{oscillation} as \emph{success criteria}.

\begin{rem}
Note that {$\infty$-\textsc{consumption}} is the only success criterion that relies on supermartingales rather than proper martingales. The reason for defining the other criteria on supermartingales is entirely semantic. We want to distinguish between strategies (martingales/supermartingales) and payoffs (success criteria). In what follows, when $\infty$-\textsc{gains} or \textsc{oscillation} are considered, it is often assumed (when no loss of generality occurs) that the supermartingales in question are in fact proper martingales.

Also, the standard definition of (super)martingales asserts non-negative values. We include the requirement that there be non-negative values in the success criteria by imposing that the martingales  do not go bankrupt. It is sometimes convenient to assume (when no loss of generality occurs) that the martingales in question are non-negative. The reason for not asserting that all martingales take non-negative values is to allow for the next definition.
\end{rem}

A (super)martingale is called \emph{history-independent} if $M'(\sigma)=M'(\tau)$, whenever $\norm{\sigma}=\norm{\tau}$. The consumption of $M$, however, may depend on history.

For $A\subset \fQ$, an \emph{$A$}-(super)martingale is a (super)martingale whose wagers take values in $A$. We will be mainly interested in restricting the wagers to the set of integers $\fZ$, and the set $V:=\set{a\in\fQ:\ \norm{a}\geq 1\text{ or }a=0}$.

We define a \emph{predictability class} (\emph{class}, for short) as a pair $\mathcal C=(A,C)$, where $A\subset \fQ$ and $C\in\set{\text{$\infty$-\textsc{gains}, $\infty$-\textsc{consumption}, \textsc{oscillation}}}$.
\begin{definition}
We say that a class $(A_1,C_1)$ \emph{implies} another class $(A_2,C_2)$ if for every $x\in\set{-1,+1}^\fN$ and every $A_1$-supermartingale $M_1$ that achieves $C_1$ on $x$, there exists an $A_2$-supermartingale $M_2$ such that
\begin{enumerate}[(a)]
\item $M_2$ achieves $C_2$ on $x$; and
\item $M_2$ is computable relative to $M_1$ (where $M_1$ is represented by listing its values in lowest common terms along an effective enumeration of $\{-1,+1\}^{<\infty}$).
\end{enumerate}
\end{definition}
Note that implication is a transitive relation.
It turns out that the classes we study exhibit the property that when they do not imply each other, they satisfy a stronger relation than just the negation of implication. 
\begin{definition}
We say that a class $(A_1,C_1)$ \emph{anti-implies} another class $(A_2,C_2)$ if there exists a computable history-independent $A_1$-supermartingale $M_1$, such that for any countable set of $A_2$-supermartingales (not necessarily computable) $\mathcal B$, there exists a sequence $x\in\set{-1,+1}^\fN$ on which
\begin{enumerate}[(a)]
\item $M_1$ achieves $C_1$; and
\item none of the elements of $\mathcal B$ achieves $C_2$.
\end{enumerate} 
\end{definition}
 
Anti-implication behaves similarly to the negation of implication in the following sense.
\begin{lem}\label{lem anti}
Let $\mathcal C_1$, $\mathcal C_2$, and $\mathcal C_3$ be classes. If $\mathcal C_2$ implies $\mathcal C_3$ and $\mathcal C_1$ anti-implies $\mathcal  C_3$, then $\mathcal C_1$ anti-implies $\mathcal C_2$.  
\end{lem}
\begin{proof}
Let $\mathcal C_i=(A_i,C_i)$, $i=1,2,3$, be the three classes of Lemma~\ref{lem anti}. Take a supermartingale $M_1$ that separates $\mathcal C_1$ from $\mathcal C_3$. Let $\mathcal B$ be a countable set of $A_2$-supermartingales. Let $\mathcal B'$ be the set of all $A_3$-supermartingales computable from some element of $\mathcal B$. Since $\mathcal B'$ is countable, there exists a sequence $x\in\set{-1,+1}^\fN$ on which $M_1$ achieves $C_1$, but no element of $\mathcal B'$ achieves $C_3$. Since $\mathcal C_2$ implies $\mathcal C_3$, no element of $\mathcal B$ achieves $C_2$ on $x$.
\end{proof}
\section{Implication results}
\begin{figure}
\begin{center}
\begin{tabular}{c |c c c c c}
Wagers& \multicolumn{5}{c}{Success Criterion}\\
& $\infty$-\textsc{gains} & & $\infty$-\textsc{consumption} & & \textsc{oscillation} \\
\hline
\\ 
$\fQ$ & $\bullet$    & $\xleftrightarrow{}$ & $\bullet$     & $\xleftrightarrow{}$ & $\bullet$\\
& $\uparrow$&&&&\\
$V$   & $\bullet$    & $\xleftrightarrow{}$   & $\bullet$     &  & $\bullet$\\
& $\uparrow$&&&&$\updownarrow$\\
$\fZ$ & $\bullet$    & $\xleftarrow{}$ & $\bullet$     & $\xleftarrow{}$ & $\bullet$\\
\\
\hline
\end{tabular}	
\end{center}
\caption{\label{fig summary}Relations between classes. Arrows indicate implication.}
\end{figure}

This section contains propositions that explain the arrows in Figure~\ref{fig summary} and also their transitive closure, by transitivity. For any success criterion $C$, if $A\subset B\subset \fQ$, then $(A,C)$ implies $(B,C)$. This explains the upwards arrows of Figure~\ref{fig summary}, since $\fZ\subset V\subset\fQ$. The leftwards arrows of Figure~\ref{fig summary} follow from the fact that $\infty\text{-\textsc{consumption}}$ implies $\infty\text{-\textsc{gains}}$ and the following proposition.
\begin{prop}\label{prop oscillate to save}
For every $A\subset\fQ$ that includes $0$, $(A,\text{\textsc{oscillation}})$ implies $(A,\infty\text{-\textsc{consumption}})$.
\end{prop}

\begin{proof}
Let $0\in A\subset\fQ$, $x\in\set{-1,+1}^\fN$, and let $M$ be an $A$-supermartingale that achieves \textsc{oscillation} on $x$. We assume w.l.o.g. that $M$ is a proper martingale, because by definition $M$ achieves \textsc{oscillation} iff $\tilde M$ oscillates. Take $a,b\in\fQ$, such that
\[
\liminf_{n\to\infty} M(x\restriction n)<a<b<\limsup_{n\to\infty} M(x\restriction n).
\]

We construct a supermartingale $S$ that achieves $\infty\text{-\textsc{consumption}}$ on $x$. In the beginning $S$ waits until $M$'s capital drops below $a$. Then $S$ mimics $M$ until $M$'s capital is above $b$. At this point $S$ consumes $b-a$ and starts over waiting until $M$'s capital drops below $a$, mimicking $M$ until $M$'s capital goes above $b$, consuming $b-a$, and starting over again. 
 
Formally, for $y\in\set{-1,+1}^\fN$, define stopping times $n_0(y),n_1(y),\ldots$ recursively by 
\begin{align*}
n_0(y)&=\inf\set{n\geq 0:\ M(y\restriction n)< a},\\
n_{2i+1}(y) &=\inf\set{n>n_{2i}:\ M(y\restriction n)> b},\\
n_{2(i+1)}(y) &=\inf\set{n>n_{2i+1}:\ M(y\restriction n)<a},
\end{align*}
with the convention that the infimum of the empty set is $\infty$.

Define an $A$-supermartingale $S=S_{a,b}$ by specifying $S'$, $S(\varepsilon)$, and $f=\tilde S-S$  as follows: $S(\varepsilon)=2a$; before time $n_0(y)$, $S'\equiv 0$ and $f\equiv 0$. For $n_{2i}(y) \leq t < n_{2i+1}(y)$, set $S'(y\restriction t)=M'(y\restriction t)$; otherwise $S'=0$. At times $\set{n_{2i+1}(y)}_{i=0}^\infty$, $f$ increases by $b-a$; otherwise $f$ doesn't change.

Since, on $x$, $M$ crosses the interval $(a,b)$ infinitely many times, all $n_i(x)$ are finite; therefore $f$ increases infinitely many times, and so $S$'s consumption on $x$ is infinite.

Note that it is not assumed that  $a$ and $b$ are computable relative to $M$. Formally, a parameterized family of supermartingales $\{S_{a,b}:a,b\in \fQ\}$ is considered. Each member of this family is computable relative to $M$, and at least one of them achieves $\infty\text{-\textsc{consumption}}$ whenever $M$ oscillates.
\end{proof}

The next proposition explains why $V$ and $\fZ$ \textsc{oscillation} are the same.
\begin{prop}\label{prop V to Z oscillate}
$(V,\text{\textsc{oscillation}})$ implies $(\set{0,-1,+1},\text{\textsc{oscillation}})$.
\end{prop}

\begin{proof}
Let $x\in\set{-1,+1}^\fN$, and let $M$ be a $V$-martingale that oscillates on $x$. Let $L=\liminf_{n\to\infty}M(x\restriction n)$. There exists $t_0\in\fN$ such that $M(x\restriction t)> L-\frac 1 2$, for every $t\geq t_0$.

We construct a parameterized family of $\set{0,-1,+1}$-martingales, $\set{S_{e,l}}_{e\in\fN,l\in\fQ}$, such that each $S_{e,l}$ is computable from $M$; and $S_{e,l}$ oscillates between 1 and 2 on $x$, whenever $e\geq t_0$ and $|l-L|<\frac 1 4$. Note that we do not assume that $t_0$ is computable from $M$ or $L$ (not even relative to $x$).

\begin{align*}
S_{e,l}(\varepsilon)&=1,\\
S_{e,l}'(y\restriction t)&=
\begin{cases} 
\sign(M'(y\restriction t)) &\text{if $t\geq e$, $\norm{M(x\restriction t)-l}< \tfrac 1 4$, and $S_{e,l}(y\restriction t)=1$,}\\
-\sign(M'(y\restriction t)) &\text{if $t\geq e$, $\norm{M(x\restriction t)-l}< \tfrac 1 4$, and $S_{e,l}(y\restriction t)=2$,}\\
0&\text{otherwise,}
\end{cases}
\end{align*}
where
\begin{equation*}
\sign(z):=\begin{cases}
+1 & \text{if $z>0$,}\\
-1 & \text{if $z<0$,}\\
0 & \text{if $z=0$.}\\
\end{cases}
\end{equation*}

Fix $e\geq t_0$ and $l\in(L- \frac 1 4,L + \frac 1 4)$. Any neighborhood of $L$ is visited by $M(x\restriction t)$ infinitely often. For $t\geq t_0$, if $|M(x\restriction t)-L|<\frac 1 2$ then $\sign(M'(x\restriction t))$ is either $0$ or $x_{t+1}$. Hence, $M(x\restriction t)$ visits $(l- \frac 1 4,l + \frac 1 4)$ infinitely often and each time it visits, $\sign(M'(x\restriction t))$ is either 0 or $x_{t+1}$. Since $M(x\restriction t)$ does not converge, $\sign(M'(x\restriction t))=x_{t+1}$ infinitely often. By the definition of $S_{e,l}$, we get that $S_{e,l}(x\restriction t)$ changes from 1 to 2 and back to 1 infinitely often.
\end{proof}

The rightwards arrows in the top row of Figure~\ref{fig summary} are explained by showing that $(\fQ,\infty\text{-\textsc{gains}})$ implies $(\fQ,\text{\textsc{oscillation}})$. This is proved in two steps: Propositions~\ref{prop Q to I gain} and \ref{prop I gain to oscillation}. Proposition~\ref{prop Q to I gain} answers a question from \citet[p. 164]{chalcraft12} in the negative.  

\begin{prop}\label{prop Q to I gain}
$(\fQ,\infty\text{-\textsc{gains}})$ implies $([-1,1]\cap\fQ,\infty\text{-\textsc{gains}})$.
\end{prop}

\begin{proof}

Let $M$ be a $\fQ$-supermartingale. As usual, we assume w.l.o.g. that $M$ is a non-negative proper martingale. We further assume that $M(\sigma)\geq 1$, for all $\sigma\in\set{-1,+1}^{<\infty}$; otherwise, consider the martingale $M+1$.

We define a $[-1,1]\cap\fQ$-martingale $S$ that makes $\infty\text{-\textsc{gains}}$ whenever $M$ does. We define $S$ by specifying its initial capital and its wagers at any $\sigma\in\set{-1,+1}^{<\infty}$ as follows:
\begin{align*}
S(\varepsilon) &=M(\varepsilon),\\
S'(\sigma) &= \frac{M'(\sigma)}{M(\sigma)}. 
\end{align*}

For any $x\in\set{-1,+1}^{\fN}$ and $t\in\fN$, we have
\begin{equation*}
\label{eq M} M(x\restriction t) = S(\varepsilon)\prod_{l=0}^{t-1}(1+x_{l+1}S'(x\restriction l)).
\end{equation*}

Since $log(1+z)\leq z$, for every $z \geq -1$, we have
\begin{multline*}
1\leq M(x\restriction t) = e^{\log(S(\varepsilon))+\sum_{l=0}^{t-1}\log(1+x_{l+1}S'(x\restriction l))}\\
\leq e^{S(\varepsilon)-1+\sum_{l=0}^{t-1}x_{l+1}S'(x\restriction l)}= e^{S(x\restriction t) -1}.
\end{multline*}

It follows that $S(x\restriction t)\geq 1$; therefore $S$ never goes bankrupt. Also, if $ M(x\restriction t)\to\infty$, as $t\to\infty$, so does $S(x\restriction t)$.
\end{proof}

\begin{prop}\label{prop I gain to oscillation}
$([-1,1]\cap\fQ,\infty\text{-\textsc{gains}})$ implies $([-1,1]\cap\fQ,\text{\textsc{oscillation}})$.
\end{prop}

\begin{proof}
Let $M$ be a $[-1,1]\cap\fQ$-martingale. As usual, $M$ is assumed to be a non-negative proper martingale. We define a $[-1,1]\cap\fQ$-martingale $S$ that oscillates whenever $M$ makes $\infty\text{-\textsc{gains}}$. 

The initial capital of $S$ is an arbitrary number $s_0>0$. Suppose $M$'s initial capital is $m_0$. In the first phase $S$ tries to gain money until its capital becomes at least 2. Let $\alpha_0=\max\set{1,\frac{m_0}{s_0}}$. During the first phase $S'(\cdot)$ is defined as $\frac{M'(\cdot)}{\alpha_0}$, so that either $\frac M S$ remains constant (if $\frac{m_0}{s_0}\geq 1$) or $S-M$ remains constant (if $\frac{m_0}{s_0} < 1$). In either case we have: $S$'s wager is bounded by 1; $S$ doesn't go bankrupt; and if $M$'s capital grows indefinitely, the first phase is bound to terminate. 

In the second phase $S$ tries to lose money. To this end, $S'(\cdot)$ is defined as $-M'(\cdot)$. The second phase terminates when $S$'s capital first drops down to some $s\leq 1$. Since $M$'s wagers are bounded by 1, $s>0$. The second phase is guaranteed to terminate as soon as $M$'s capital grows sufficiently.

The two phases are repeated starting with the capital in $(0,1]$, going above 2, and returning to $(0,1]$ again. As long as $M$'s capital grows indefinitely, the two phases are bound to terminate and, therefore to be repeated infinitely many times; thus $S$ oscillates.
\end{proof}

The remaining rightwards arrow in the middle row of Figure~\ref{fig summary} is explained. 
\begin{prop}\label{prop V gain to save}
$(V,\infty\text{-\textsc{gains}})$ implies $(V,\infty\text{-\textsc{consumption}})$.
\end{prop}

\begin{proof}
Let $M$ be a $V$-supermartingale. Assume without loss of generality that $M$ is a proper martingale and $M(\varepsilon)\geq 2$. 

We define a supermartingale $S$ that achieves $\infty\text{-\textsc{consumption}}$ whenever $M$ achieves $\infty\text{-\textsc{gains}}$. The initial capital of $S$ is twice the initial capital of $M$. $S$ bets proportionally to $M$, and consumes 1 every time $M$ doubles its capital.

That is, $S(\varepsilon)=2M(\varepsilon)$, $S'(\sigma)=M'(\sigma)\frac{S(\sigma)}{M(\sigma)}$, for every $\sigma\in\set{-1,+1}^{<\infty}$. This ensures that the ratio between the capital of $S$ and $M$ stays constant as long as $S$ does not consume. Every time $M$ doubles its capital $S$ consumes an amount of 1 and as a result the ratio between $S$'s and $M$'s capital decreases by $\frac{1}{M(\sigma)}$. As long as the ratio is at least 1, $S'(\sigma)\in V$, as required.  

It remains to show that $\frac{S(\sigma)}{M(\sigma)}\geq 1$, for every $\sigma\in\set{-1,+1}^{<\infty}$. For $\sigma\in\set{-1,+1}^{<\infty}$, suppose $M$ doubles its capital $k$ times along $\sigma$ at the prefixes $\sigma_1,\ldots,\sigma_k$. That is, $\sigma_i$ is the shortest prefix of $\sigma$ such that $M(\sigma_i)\geq 2M(\sigma_{i-1})$, for all $i=1,\ldots k$, where $\sigma_0=\varepsilon$. By induction on $k$, $\frac{S(\sigma)}{M(\sigma)}=2-\sum_{i=1}^k\frac 1 {M(\sigma_{i-1})}\geq 2 - \sum_{i=1}^k\frac 1 {2^i} \geq 1$.
\end{proof}

\section{Anti-implication results}
Even more interesting than the implication results are the anti-implication results. By virtue of Lemma~\ref{lem anti}, we need only to separate adjacent strongly connected components of the diagram in Figure~\ref{fig summary}, and consider one representative from each strongly connected component.

The next theorem separates between integer $\infty$-\textsc{gains} and $\infty$-\textsc{consumption}. 
\begin{thm}[\citealt{teutsch2013}]\label{thm Z gain not to save}
$(\set{1},\infty\text{-\textsc{gains}})$ anti-implies $(\fZ,\infty\text{-\textsc{consumption}})$.
\end{thm}
The proof follows the main lines of Teutsch's original proof while correcting a minor flaw.\footnote{With the notation of the original proof \citep[p. 150]{teutsch2013}, a failure may occur when $S$'s wager is $q-1$, $r=1$, and $X(n+1)=\mathtt t$. In this case $q'=q+1$ (and $r'=m'=m-1$), violating the supposed invariants, Things I and II.} We offer here a concise correction. The reader is referred to the original proof for a detailed exposition. 
\begin{proof} 
Let $\{S_e\}_{e=1}^\infty$ be an arbitrary sequence of $\fZ$-supermartingales. Let $M$ be the $\set{1}$-martingale with initial capital $1$. We construct a sequence $x\in\set{+1,-1}^{\fN}$ on which $M$ makes $\infty$-\textsc{gains} and none of the $\{S_e\}_{e=1}^\infty$ makes $\infty$-\textsc{consumption}.

We may assume w.l.o.g. that $\{S_e\}_{e=1}^\infty$ are integer-valued, because if they were not so, then $\set{\lceil S_e\rceil}_{e=1}^\infty $ would be. Also, we may include among $\{S_e\}_{e=1}^\infty$ constant martingales of arbitrarily large capital, and further assume w.l.o.g. that $\{S_e\}_{e=1}^\infty$ are non-negative. We define a sequence $x$ recursively. Assume $x\restriction n$ is already defined. Define the following integer-valued functions:
\begin{align*}
s_e(n)&:= S_e(x\restriction n),\\
s'_e(n)&:= S'_e(x\restriction n),\\
f_e(n)&:= \tilde S_e(x\restriction n)-S_e(x\restriction n),\\
m_1(n)&:= M(x\restriction n),\\
m_{e+1}(n)&:= m_e(n)-r_e(n), \\
\shortintertext{where}
q_e(n)&:=\left\lfloor \frac{s_e(n)}{m_e(n)}\right\rfloor,\\
r_e(n)&:=s_e(n)-q_e(n)m_e(n).
\end{align*}

The above is well defined as long as $m_1(n)> 0$, and in that case we have 
\begin{equation}\label{eq m}
m_1(n)\geq m_2(n)\geq\cdots \geq 1.
\end{equation}
The definition of $m_1,m_2,\ldots$ is the main departure from Teutsch's construction, where the same $m$ is used against all of the $S_e$s.
 
Let $i=i(n)=\min\set{e:s'_e(n)\neq q_e(n)}$. Note that the set in the definition of $i$ is not empty, since $\set{m_e(n)}_{e=1}^\infty$ is bounded (by $m_1(n)$) and $\set{s_{e}(n):s'_e(n)=0}$ is not bounded, because $\{S_e\}_{e=1}^\infty$ include arbitrarily large constant martingales.  

We are now ready to define 
\[
x_{n+1}=
\begin{cases}
+1 &\text{if $s'_{i}(n)< q_i(n)$,}\\
-1 &\text{if $s'_{i}(n)> q_i(n)$.}
\end{cases}
\]
The following properties follow by induction on $n$:
\begin{enumerate}[(i)]
\item $m_1(n)>0$ (hence $m_1(n)\geq m_2(n)\geq\cdots \geq 1$, by \eqref{eq m});
\item for every $e < i(n)$, the pair $\langle q_e(n+1),r_e(n+1) \rangle$ is lexicographically not greater than $\langle q_e(n),r_e(n) \rangle$, with strict inequality if $S_e$ consumes money at time $n$;
\item $\langle q_i(n+1),r_i(n+1) \rangle$ is lexicographically strictly less than $\langle q_i(n),r_i(n) \rangle$.
\end{enumerate}

The only delicate point to notice when verifying (i)--(iii) is (ii), in the case where $S_e$ does not consume and $r_e(n)=m_e(n)-1$ (this is where Teutsch's proof fails). Before proving the delicate case, we first explain how the proof of Theorem~\ref{thm Z gain not to save} is concluded by assuming (i)--(iii).

From (ii) and (iii) the sequence of finite sequences
\[
\set{\langle q_e(n),r_e(n)\rangle_{e=1}^{i(n)}}_{n=1}^\infty
\]
is strictly decreasing. Since natural numbers cannot decrease indefinitely, it must be the case that $\lim_{n\to\infty}i(n)=\infty$ and each $\langle q_e(n),r_e(n)\rangle$ is fixed for $n$ large enough. It follows from (ii) that none of $\set{S_e}_{e=1}^\infty$ achieves $\infty$-\textsc{consumption}. Since $\lim_{n\to\infty}i(n)=\infty$ and $\set{S_e}_{e=1}^\infty$ include arbitrarily large constants, $M$'s capital must go to $\infty$.

It remains to verify (ii) in the case where $S_e$ does not consume and $r_e(n)=m_e(n)-1$. In this case, if $x_{n+1}=-1$, we would have $q_e(n+1)=q_e(n)+1$ (and $r_e(n+1)=0$). Fortunately, this situation is avoided by the definition of $m_e(n)$. If $r_e(n)=m_e(n)-1$, then $m_{e+1}(n)=1$; hence, by \eqref{eq m}, $m_i(n)=1$; hence, by the definition of $x$ and the assumption that $S_i$ is non-negative integer-valued, we get $x_{n+1}=+1$.
\end{proof}

The next theorem explains the separation between integer $\infty$-\textsc{consumption} and \textsc{oscillation}.
\begin{thm}\label{thm V save not to oscillate}
$(\set{1},\infty\text{-\textsc{consumption}})$ anti-implies $(\fZ,\text{\textsc{oscillation}})$.
\end{thm}

\begin{proof}
Let $M$ be the $\set{1}$-martingale with initial capital $1$. That is, 
\[
M(\sigma) =1+\sum_{k=1}^n \sigma_k\ ,
\]
for any non-negative integer $n$ and $\sigma\in\set{-1,+1}^n$.

Define a consumption function
\begin{align*}
f(\sigma)&= \left\lfloor\frac 1 2 \max_{0\leq k\leq n}M(\sigma\restriction k)\right\rfloor,\\
\shortintertext{and a supermartingale}
m(\sigma)&=M(\sigma)-f(\sigma)
\end{align*}
that will provide the separation.

Note that $M$,$m$ and $f$ are defined such that they all diverge to $\infty$ on the same set of  sequences and $f$ increases only after two consecutive $+1$s and never increases in two consecutive periods.

Let $\mathcal B$ be a countable set of $\fZ$-martingales. Assume w.l.o.g.\ that the members of $\mathcal B$ are non-negative proper martingales with initial integer values. When a $\fZ$-martingale oscillates there is an integer between its limits superior and inferior; thus we can break the task of oscillation into countably many smaller tasks of oscillating around a given integer, and distribute these tasks among countably many copies of that martingale. Formally, we associate each pair $(S,k)\in\mathcal B\times \fN$ with the goal of achieving 
\[
\liminf_{n\to\infty} S(x\restriction n)\leq k<\limsup_{n\to\infty} S(x\restriction n). 
\]

For convenience, the elements of $\mathcal B\times \fN$ are arranged in a sequence $\set{(S_e,k_e)}_{e=1}^\infty$, such that whenever $e<e'$ and $S_e=S_{e'}$, then $k_e<k_{e'}$. 

We define a sequence $x\in\set{-1,+1}^\fN$ recursively. Assume $x\restriction n$ is defined for some $n\geq 0$. We say that $(S_e,k_e)$ \emph{receives attention} at time $n$, if $e$ is minimal with respect to the following properties:
\begin{enumerate}[(i)]
\item $S_e(x\restriction n)\leq k_e$,
\item $m(x\restriction n) > e$,
\item $S'_e(x\restriction n)\neq 0$.
\end{enumerate}
Define
\[
x_{n+1}= \begin{cases}
-\sign S'_e(x\restriction n)& \text{if some $(S_e,k_e)$ receives attention at time $n$,}\\
+1&\text{otherwise.}
\end{cases}
\]

By (ii), no $(S_e,k_e)$ receives attention at time $n$ when $m(x\restriction n)=1$; therefore, $m(x\restriction t)\geq 1$, for every $t$.
 
Let $L$ be an arbitrary integer satisfying $1\leq L \leq \liminf_{n\to\infty} m(x\restriction n)$. Since $m$ is integer-valued there exists some $n_0\in\fN$ such that $m(x\restriction n)\geq L$, for every $n\geq n_0$. Note that $m(x\restriction n+1)=m(x\restriction n)$ only if $x_{n+1}=1$ and $f(x\restriction n+1)=f(x\restriction n)+1$. Since $f$ never increases over two consecutive periods, we never have $m(x\restriction n)=m(x\restriction n+1)=m(x\restriction n+2)$, and, therefore, the set of indexes $I=\set{n\geq n_0:\\m(x\restriction n)> L}$ is infinite. 

Consider the sequence of $L$-tuples of integers $\set{a_n}_{n\in I}$ defined by
\[
a_n=\tuple{\min\set{k_e,S_e(x\restriction n)}}_{e=1}^L.
\]  

The proof will be concluded by showing that 
\begin{itemize}
\item[$(*)$] $a_n$ is non-increasing with respect to the lexicographic order on $\fZ^L$, and 
\item[$(**)$]it decreases whenever $m(x\restriction n+1)=L$.
\end{itemize}
It follows from $(*)$ that $a_n$ is fixed for $n$ large enough, and from $(**)$ that $m(x\restriction n)=L$ only finitely many times. By showing it for every $1\leq L \leq \liminf m(x\restriction n)$, we prove that $\liminf m(x\restriction n)=\infty$; therefore $m$ achieves $\infty$-\textsc{consumption}. Furthermore, for every $S\in\mathcal B$ and $k\in\fN$, $S$ does not oscillate around $k$ (by taking $L$ so large that $(S,k)\in\set{(S_1,k_1),\ldots,(S_L,k_L)}$); therefore none of the martingales in $\mathcal B$ oscillates.

It remains to prove $(*)$ and $(**)$. Let $n\in I$. That is, $m(x\restriction n)>L$. If $m(x\restriction n+1)>L$ then the successor of $n$ in $I$ is $n+1$ and according to our definition of $x$, $a_{n+1}$ is not lexicographically greater than $a_n$. If $m(x\restriction n+1)=L$ then $m(x\restriction n)=m(x\restriction n+2)=L+1$, since $f$ increases only after two consecutive $+1$s; that is, $n+2$ is $n$'s successor in $I$. By the definition of $x$, $a_{n+1}$ is strictly less than $a_n$, and so it remains to show that $a_{n+2}$ is not greater than $a_{n+1}$.

If some $(S_e,k_e)$ receives attention at time $n+1$, then $a_{n+2}$ is less than $a_{n+1}$. If not, then the only thing that can make $a_{n+2}$ increase is that $S'_L(x\restriction n+1)>0$ and $S_L(x\restriction n+1)<k_L$, but this is not possible because: (a) if $k_L=1$, then $S_L(x\restriction n+1)=0$, and so $S'_L(x\restriction n+1)=0$; (b) if $k_L>1$, then there is some $e<L$ such that $S_e=S_L$ and $k_{e}=k_L-1$, and so $(S_e,k_e)$ (or a smaller index) receives attention at time $n+1$. 
\end{proof}

A very simple history-independent strategy, betting $\frac 1 n$ at time $n$, separates between $\fQ$ and $V$ $\infty$-\textsc{gains}. 
\begin{thm}\label{thm R gain not to V gain}
$(\set{\frac 1 n}_{n=1}^\infty,\infty\text{-\textsc{gains}})$ anti-implies $(V,\infty\text{-\textsc{gains}})$.
\end{thm}

The proof of Theorem~\ref{thm R gain not to V gain} is probabilistic.\footnote{A working paper by \cite{bavly-peretz} provides a constructive proof.} In probability theory, martingales have a slightly different meaning than the algorithmic randomness meaning that is used in the present paper. To distinguish between the two, we call a \emph{martingale process} a sequence of integrable random variables $\set{X_n}_{n=0}^\infty$ adapted to a filtration $\mathcal F_0\subset \mathcal F_1\subset \cdots \subset \mathcal F$, such that
\[
\mathbb E[X_{n+1}|\mathcal F_n]=X_n, \quad\text{for all }n=0,1,2,\ldots
\]

We will utilize the following standard results from probability theory. 
\begin{thm}[Doob's martingale convergence]
Let $(X_n,\mathcal F_n)_{n=0}^\infty$ be a martingale process. If $\sup \set{\mathbb E|X_n|}_{n=0}^\infty<\infty$, then $\lim\limits_{n\to\infty} X_n=X_\infty$ exists (almost surely) and $\mathbb E |X_{\infty}|<\infty$.
\end{thm}
\begin{cor}[bounded second moment]
Let $(X_n,\mathcal F_n)_{n=0}^\infty$ be a martingale process. If $sup\set{\mathbb E(X_n)^2}_{n=0}^\infty<\infty$, then $\lim\limits_{n\to\infty} X_n$ exists and is finite almost surely.
\end{cor}
\begin{cor}[non-negative]
Let $(X_n,\mathcal F_n)_{n=0}^\infty$ be a martingale process. If all $X_n$ are non-negative (almost surely), then $\lim\limits_{n\to\infty} X_n$ exists and is finite almost surely.
\end{cor}
The interested reader is referred to \citet[Chapter 4]{shiryaev} for proofs and further discussion.

\begin{proof}[Proof of Theorem \ref{thm R gain not to V gain}.]
The separating $\set{\frac 1 n}_{n=1}^\infty$-martingale is very simple: it bets $\frac 1 n$ on $+1$ at time $n$. The more sophisticated part is constructing the separating binary sequence (given a countable set of $V$-martingales) and setting the initial value large enough (independently of the $V$-martingales).
 
Let $\epsilon_1,\epsilon_2,\dots$ be independent random variables assuming the values $\pm 1$ with probability $(\frac 1 2,\frac 1 2)$. The series $\sum_{i=1}^\infty \frac{\epsilon_i}{i}$ converges a.s., by Doob's martingale convergence theorem, since the finite sums have bounded second moments. Let $L>0$ be large enough that 
\begin{equation}\label{eq L}
\Pr\left(\norm{\sum_{i=N}^{N+K} \frac{\epsilon_i}{i}}< L,\ \forall N,K\in\fN\right)>0.
\end{equation}
Note that $L$ is a universal constant; it does not depend on the realization of $(\epsilon_1,\epsilon_2,\ldots)$.
 
Define a history-independent $\set{\frac 1 n}_{n=1}^\infty$-martingale  $S_0:\set{-1,+1}^{<\infty}\to\fQ$ to be 
\[
S_0(x_1,\ldots,x_n)=(L+2)+\sum_{i=1}^n \frac {x_i} i.
\]

Let $\set{S_1,S_2,\ldots}$ be a countable set of $V$-supermartingales. Assume w.l.o.g. that $S_i(\sigma)\geq 0$, for every $i\geq 1$ and  $\sigma\in\set{-1,+1}^{<\infty}$.

We shall construct a random process $x_1,x_2,\ldots$, adapted to $\epsilon_1,\epsilon_2,\ldots$ (namely, each $x_n$ is a function of $\epsilon_1,\ldots,\epsilon_n$) that satisfies the following properties:
\begin{enumerate}[(i)]
\item $\limsup\limits_{n\to\infty} S_j(x\restriction n)<\infty$, for all $j\geq 1$ (almost surely);
\item $\lim\limits_{n\to\infty} S_0(x\restriction n)=\infty$ (almost surely);
\item $\Pr(\inf_{n\in\fN} S_0(x\restriction n) \geq 1)>0$.
\end{enumerate}

The definition of $x_1,x_2,\ldots$ relies on an increasing sequence of stopping times $0=n_0(x)< n_1(x)<\dots$ defined recursively, for any realization of $x\in\set{-1,+1}^{\fN}$. In odd intervals, between times $n_{2i}$ and $n_{2i+1}$, $x_n$ is chosen randomly by setting it to $\epsilon_n$. In even intervals, between times $n_{2i+1}$ and $n_{2i+2}$, $x_n$ is determined as a function of $x_1\ldots,x_{n-1}$. In both cases we make sure that given that $n_i$ is finite, $n_{i+1}$ will be finite as well (almost surely). 

The purpose of the random intervals is to gain time. The martingale convergence theorem ensures that none of the martingales change much during each random interval. As time progresses $S_0$ gains an advantage over the $V$-martingales, since its wagers become smaller; and so it can endure more losses. By the end of the $i$th random interval, $S_0$ can incur more losses than $S_1,\ldots,S_i$ together without going bankrupt. This is achieved by requiring that $\sum_{j=1}^iS_j(x\restriction n_i)<n_i$, and so $S_0$ never accumulates losses of more than 1 during a deterministic interval.

During the $i$th deterministic interval, the bits of $x$ prevent $\tuple{S_1,\ldots,S_i}$ from increasing (in the lexicographic order). At the same time, $S_0$ may incur losses, but it does not go bankrupt thanks to the advantage it gained during the preceding random interval. Since non-negative $V$-martingales cannot decrease indefinitely, they must stop betting at some point. When none of $S_1,\ldots,S_i$ bets, $x_n$ is set to be $+1$, and so it is guaranteed that $S_0$ will eventually regain its losses plus a positive amount of $L$, which is the stopping condition of the deterministic interval. 

The choice of $L$ was made so that there would be an event of positive probability in which $S_0$ never accumulates losses of more than $L$ in any random interval. In the proceeding deterministic interval $S_0$ can accumulate additional losses of at most 1, and by the end of the deterministic interval it re-gains all its losses from the two time intervals (Figure~\ref{fig S_0}). By setting $S_0(\varepsilon)=2+L$, it is guaranteed that $S_0$ never goes below 1 and, therefore, does not go bankrupt (verifying (iii)).

\begin{figure}
\begin{center}
\includegraphics[clip=true, trim= 3cm 17cm 3cm 8cm, width=\textwidth]{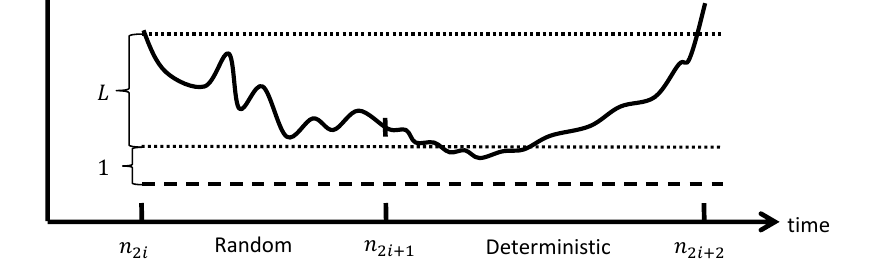}
\caption{\label{fig S_0}A typical $S_0$ path during two consecutive time intervals showing that $S_0$ never drops by more than $L+1$ during that period and always re-gains its losses by the end of the period.}
\end{center}
\end{figure}

Finally, consider the accumulated gains. The martingale convergence theorem guarantees that the accumulated gains (or loses) during the random intervals converge, for all of the martingales. The deterministic intervals are designed so that the accumulated gains of each of the $V$-martingales are eventually non-increasing, whereas the gains of $S_0$ increase by at least $L>0$ in every deterministic interval (verifying (ii)).

We now formalize the above arguments. Let $n_0\equiv 0$. For $i>0$, let
\begin{align*}
n_{2i-1}(x)&=\inf\set{t>n_{2i-2}(x):\ \sum_{j=1}^iS_j(x\restriction t)<t}\text{ and}\\
n_{2i}(x)&=\inf\set{t>n_{2i-1}(x):\ \sum_{n=n_{2i-1}(x)}^t \frac{x_n} n \geq L},
\end{align*}
with the convention that the infimum of the empty set is $+\infty$.

For $i\in\fN\cup\set{0}$ and $n_{2i}(x)< n \leq n_{2i+1}(x)$, set $x_{n}=\epsilon_n$ (note that $n_i$ are stopping times; hence $i$ is a function of $x\restriction n-1$).

For $n_{2i+1}(x) < n \leq n_{2i+2}(x)$, the value of $x_{n}$ will be determined by $x_1,\ldots,x_{n-1}$, as follows.

Let
\[
j=j(n)=\inf\set{e:\ \exists t,\,n_{2i+1}(x)\leq t\leq n,\ S_e'(x_1,\ldots,x_t)\neq 0}.
\]
With $i$ and $j$ as given above, define
\[
x_{n+1}=
\begin{cases}
-1 &\text{if $j\leq i+1$ (in particular, $j$ is finite) and $S_j'(x_1,\ldots,x_n)>0$,}\\
+1 &\text{otherwise.}
\end{cases}
\]

Let us call the union of all random/deterministic intervals the \emph{random/ deterministic phase}. Express $S_j$ ($j=0,1,\ldots$) as a sum of two martingales, $S_j^{r}+S_j^{d}$, where $S_j^r$ is active only during the random phase and $S_j^d$ only during the deterministic phase. 

Formally, define $T^{r}(x)=\set{n:\exists i,\, n_{2i}(x)<n\leq n_{2i+1}(x)}$ and $T^{d}(x)=\fN\setminus T^d$, and define $S_j^{r},S_j^{d}:\set{-1,+1}^{<\infty}\to\fQ$, for $j\geq 0$, recursively by
\begin{align*}
S_j^{r}(\varepsilon)&=S_j(\varepsilon),\\
S_j^{r}(x\restriction n)-S_j^{r}(x\restriction n-1) &= 
\begin{cases}
S_j(x\restriction n)-S_j(x\restriction n-1),&\text{if $n\in T^{r}$,}\\
0                                          &\text{if $n\in T^{d}$,}\\
\end{cases}\\
\shortintertext{and}
S_j^{d} &= S_j-S_j^{r}.
\end{align*}

By Doob's martingale convergence theorem (using that the $S_j^r$s are non-negative martingale processes, for $j>0$, and $S_0^r$ has bounded second moments), for almost every $x$,
\begin{align}\label{eq S^r}
\forall j\geq 0\ \lim_{n\to\infty}S^{r}_j(x\restriction n)<\infty &&\text{(and the limits exist).}
\end{align}
Also, if $n_{2i}(x)$ is finite, then so is $n_{2i+1}(x)$ (almost surely). 

We turn now to analyze the deterministic phase. Denote the $i$th deterministic time interval by $T^d_i(x):= \set{n: n_{2i+1}(x) <n\leq n_{2i+2}(x)}$. The definition of $x$ in the deterministic phase is such that $\set{\tuple{S_j(x\restriction n)}_{j=1}^i}_{n\in T_i^d}$ is non-increasing (lexicographically, as $n$ grows), and it decreases each time $x_n=-1$; therefore it can decrease at most $\sum_{j=1}^{i+1}S_j(x\restriction n_{2i+1}(x))\leq n_{2i+1}(x)$ times, and so, if $n_{2i+1}(x)$ is finite, so is $n_{2i+2}(x)$. Since, for every $i\in\fN$, $\tuple{\floor{S^d_j(x\restriction n)}}_{j=1}^i$ is lexicographically eventually non-increasing (specifically, for every $n\geq n_{2i+1}$) and since natural numbers cannot decrease indefinitely, each one of the $S^{d}_j$s is eventually non-increasing (for every $j\geq 1$) and hence convergent, and 
\[
\lim\limits_{n\to\infty} S_j(x\restriction n)=\lim\limits_{n\to\infty} S_j^r(x\restriction n)+\lim\limits_{n\to\infty} S_j^d(x\restriction n)<\infty,\quad\text{for all $j\geq 1$,}
\]
verifying (i).

By the definition of $n_{2i+2}$, 
\begin{equation}\label{eq S^d}
S_0^{d}(x\restriction n_{2i+2})-S_0^{d}(x\restriction n_{2i})=S_0^{d}(x\restriction n_{2i+2})-S_0^{d}(x\restriction n_{2i+1})\geq L>0.
\end{equation}
Since the number of periods $n\in T_i^d$ in which $x_n=-1$ is at most $n_{2i+1}(x)$, $S_0^{d}(x\restriction n) - S_0^{d}(x\restriction n_{2i})\geq -1$, for every $n_{2i}<n\leq n_{2i+2}$; therefore 
\begin{align}\label{eq S0d}
S_0^{d}(x\restriction n)\geq iL-1,&&\text{whenever $n\geq n_{2i}$.}
\end{align}

By \eqref{eq L} there is an event of positive probability, in which 
\begin{align}\label{eq S0r}
S_0^{r}(x\restriction n)-S_0^r(\varepsilon)\geq -(i+1)L,&&\text{whenever $n\geq n_{2i}$.}
\end{align}
Combining \eqref{eq S0d}, \eqref{eq S0r}, and the fact that $S_0^r(\varepsilon)=S_0(\varepsilon)=L+2$ yields $\inf_{n}S_0(x\restriction n)\geq 1$, which concludes (iii).

Finally, $\lim_{n\to\infty}S^d_0(x\restriction n)=\infty$ by \eqref{eq S0d} and $\sup_{n\in\mathbb N}|S^r_0(x\restriction n)|<\infty$ by \eqref{eq S^r}; therefore $\lim S_0(x\restriction n)\to \infty$, as $n\to\infty$ (almost surely), concluding (ii) and the whole proof.
\end{proof}

The separation between $V$ and $\fZ$ $\infty$-\textsc{gains} is made by betting $1+\frac 1 n$ at time $n$.
\begin{thm}\label{thm V gain not to Z gain}
$(\set{1+\frac 1 n}_{n=1}^\infty,\infty\text{-\textsc{gains}})$ anti-implies $(\fZ,\infty\text{-\textsc{gains}})$.
\end{thm}

\begin{proof}
Define a history-independent $\set{1+\frac 1 n}_{n=1}^\infty$-martingale $M$ by
\begin{align*}
M(\varepsilon) = 2,\\
M'(\cdot\restriction n)=1+\frac 1 n. 
\end{align*}
Denote the harmonic sum $H(n)=\sum_{k=1}^n\frac 1 k$. For $y\in\set{-1,+1}^\fN$, we define the number of consecutive losses that $M$ can incur before going bankrupt after betting against $y\restriction n$ as 
\begin{align*}
K(y\restriction n)&=k_n(M(y\restriction n))\text{, where}\\
k_n(m)&= \max\set{k: k + H(n+k)-H(n) \leq m}.
\end{align*}

The function $K$ turns out to be a useful currency for measuring $M$'s capital. It is an integer-valued submartingale ($-K$ is a supermartingale). It decreases by 1 upon seeing $-1$, increases by at least 1 upon seeing $+1$, and increases by $n+\omega(1)$ upon seeing $n$ consecutive +1s. 

The above $\omega(1)$ can be interpreted as a production function (negative consumption function). Similar to the way in which the $\{1\}$-martingale anti-implies  $(\fZ,\infty\text{-\textsc{consumption}})$, $K$ anti-implies $(\fZ,\infty\text{-\textsc{gains}})$, except that here $K$'s production plays the role of consumption in \cite{teutsch2013}.  

Let us begin with an informal sketch of the proof. Let $\set{S_1,S_2,\ldots}$ be a countable set of $\fZ$-supermartingales. The unboundedness of $K$'s production ensures that even if $S_1$ bets proportionally to $K$ (i.e., $S'_1=S_1/K$) the ratio $S_1/K$ can be forced downwards until $K>S_1$ eventually. At this stage $M$ has sufficient capital to prevent $S_1$ from making gains and this becomes the primary goal. The secondary goal, which is pursued whenever $S_1$ bets 0, is to minimize the ratio $S_2/(K-S_1)$ until $K>S_1+S_2$. Inductively, stage $e$ begins when $K>S_1+\cdots +S_{e-1}$. The primary goal is to prevent $S_1,\ldots,S_{e-1}$ from making gains. Once the primary goal is achieved, the secondary goal is pursued, minimizing the ratio $S_e/(K-S_1-\cdots -S_{e-1})$ until $K>S_1+\cdots+S_e$.

We now formalize the above idea. As usual, assume w.l.o.g. that $S_i(\sigma)\geq 0$ for every $i\geq 1$, and  $\sigma\in\set{-1,+1}^{<\infty}$. Assume, also (for technical reasons that will become clear), without loss of generality that $\set{S_1,S_2,\ldots}$ include constant martingales of arbitrarily large capital.  

We define a sequence $x\in\set{-1,+1}^\fN$ recursively. Assume by induction that $x\restriction n$ is already defined. Let
\begin{align*}
e(n)&=\min\set{e:\ K(x\restriction n) \leq  S_1(x\restriction n)+\cdots+S_{e}(x\restriction n)}\text{, and}\\\
k(n)&=K(x\restriction n)-\left(S_1(x\restriction n)+\cdots+S_{e(n)-1}(x\restriction n)\right).
\end{align*}
By the assumption that $\set{S_1,S_2\ldots}$ include arbitrarily large constants, $e(n)$ is well defined. If there is an index $1\leq i< e(n)$ such that $S'_i(x\restriction n)\neq 0$, let $j$ be the minimum of these indexes, and let
\[
x_{n+1}=-\sign S'_j(x\restriction n).
\]
Note that in this case 
\begin{equation}\label{eq e > j}
e(n+1)> j.
\end{equation}
Assume now that $S'_1(x\restriction n)=\cdots=S'_{e(n)-1}(x\restriction n)=0$ (or $e(n)=1$). Define
\[
x_{n+1}=
\begin{cases} +1 &\text{if $S'_{e(n)}(x\restriction n)\leq \frac {S_{e(n)}(x\restriction n)}{k(n)}$,}\\
-1 & \text{if $S'_{e(n)}(x\restriction n)> \frac {S_{e(n)}(x\restriction n)}{k(n)}$.}
\end{cases}
\]
Note that in this case 
\begin{equation}\label{eq e > e}
e(n+1)\geq e(n).
\end{equation}

Since $\set{S_1,S_2,\ldots}$ include constant martingales of arbitrarily large capital, it is sufficient to prove that $\lim\limits_{n\to\infty}e(n)=\infty$ and $\lim\limits_{n\to\infty}S_i(x\restriction n)<\infty$ for every $i\in\fN$ (and that, in particular, the limits exist). 

Assume for the sake of contradiction that 
\[
e=\inf \left(\set{i:S_i(x\restriction n)\text{ diverges}}\cup\set{\liminf_{n\to\infty} e(n)}\right) < \infty.
\]
By the definition of $x$, for every $k<\liminf\limits_{n\to\infty}e(n)$, the sequence of $k$-tuples $a_n=\tuple{S_i(x\restriction n)}_{i=1}^{k}$ is eventually non-increasing in the lexicographic order; therefore it stabilizes, for $n$ large enough; therefore it remains to show that $\lim\limits_{n\to\infty}e(n)=\infty$ (since $e=\liminf\limits_{n\to\infty}e(n)$).

Take $n_0\in\fN$ such that $\forall n\geq n_0$, $S'_1(x\restriction n)=\cdots=S'_{e-1}(x\restriction n)=0$. By \eqref{eq e > j} and \eqref{eq e > e}, if $e(n)>e$, then so is $e(n+1)$; therefore $e(n)=e$, for all $n\geq n_0$. 

Denote the long division with a remainder of $\frac {S_{e}(x\restriction n)}{k(n)}$ by
\[
S_{e}(x\restriction n)=q(n)k(n)+r(n).
\]
For all $n\geq n_0$, we have $(q(n+1),r(n+1))\leq (q(n),r(n))$ (lexicographically) with a strict inequality if either $x_{n+1}=-1$ or $k(n+1)>k(n)+1$; therefore, for some $n_1\geq n_0$ and every $n\geq n_1$, $x_{n+1}=+1$ and $k(n+1)=k(n)+1$. We show that this is impossible by showing that $k(n_1)+l\ll k(n_1+l)$, as $l\to\infty$ (where $f(l)\ll g(l)$ means that $\lim_{l\to\infty}[g(l)-f(l)]=\infty$). 

Since $S_1,\ldots,S_{e-1}$ are constant after time $n_1$, it is sufficient to show that $K(x\restriction n_1)+l\ll K(x\restriction n_1+l)$ (as $l\to\infty$).
Let $m=M(x\restriction n_1)$ and $k=K(x\restriction n_1)$. We have 
\[
M(x\restriction n_1+l)= m+l + H(n_1+l)-H(n_1),
\]
and so
\begin{multline*}
K(x\restriction n_1+l)=\\
\max\set{k':k'+H(n_1 + l + k')-H(n_1 + l) \leq m+l + H(n_1+l)-H(n_1)}.
\end{multline*}
Plugging in $k'=k+l$, it suffices to show
\[
k+l+H(n_1+k+2l)-H(n_1+l)\ll m+l + H(n_1+l)-H(n_1).
\]
By eliminating constant terms and rearranging, we must show 
\[
H(n_1+k(n_1)+2l)-2H(n_1+l)\to -\infty,\quad\text{as $l\to\infty$.}
\]
The above follows from the estimate $\log(n-1)\leq H(n)\leq \log(n+1)$.
\end{proof}

\section*{Acknowledgments}
The author wishes to thank two anonymous referees for many useful comments that improved the presentation of the paper substantially, Aviv Keren and Gilad Bavly for comments and suggestions, and Andy Lewis-Pye for patiently explaining concepts of computability. The early stages of this work were done in 2010--12 while the author was a postdoctoral researcher at Tel Aviv University. The author wishes to thank his hosts Ehud Lehrer and Eilon Solan.

\end{document}